\documentclass[10pt]{article}
\usepackage{amsmath,amssymb,amsthm,amsfonts}

\newcommand{\cb}{{\boldsymbol{c}}}
\newcommand{\db}{{\boldsymbol{d}}}

\newcommand{\ub}{{\boldsymbol{u}}}
\newcommand{\vb}{{\boldsymbol{v}}}

\newcommand{\Ab}{{\boldsymbol{A}}}
\newcommand{\Bb}{{\boldsymbol{B}}}
\newcommand{\Cb}{{\boldsymbol{C}}}
\newcommand{\Db}{{\boldsymbol{D}}}

\newcommand{\cA}{{\cal A}}

\newcommand{\cE}{{\cal E}}

\newcommand{\CC}{{\mathbb C}}

 \newcommand{\NN}{{\mathbb N}}

 \newcommand{\RR}{{\mathbb R}}

 \newcommand{\ZZ}{{\mathbb Z}}

\newcommand{\supp}{{\mathop{\rm supp}\,}}

\def\bdelta{\mbox{\boldmath $\delta$}}
\newtheorem{Theorem}{Theorem}[section]

\newtheorem{Example}[Theorem]{Example}

\newtheorem{Lemma}[Theorem]{Lemma}
\newtheorem{Remark}[Theorem]{Remark}
\newtheorem{Proposition}[Theorem]{Proposition}

\begin{document}

\begin{titlepage}

\title{ Numerical methods for checking the regularity of subdivision schemes}

\author{Maria Charina}

\end{titlepage}

\maketitle

\def\shorttitle{ Numerical methods for checking the regularity of subdivision schemes}

\def\shortauthor{M.~Charina, TU Dortmund, Germany}

\begin{abstract}
 In this paper, motivated by applications in computer graphics and animation, we study the numerical
 methods for checking $C^k-$regularity of vector multivariate  subdivision schemes with dilation $2I$.
 These numerical methods arise from the joint spectral radius and restricted spectral radius approaches,
 which were shown in \cite{C2011} to characterize $W^k_p-$regularity of subdivision in terms of the same quantity. Namely, the $(k,p)-$joint spectral radius and  the $(k,p)-$restricted 
 spectral radius are equal. We show that the corresponding numerical methods in the univariate
 scalar and vector cases even yield the same upper estimate for the $(k,\infty)-$joint spectral radius for
 a certain choice of a matrix norm. The difference between the two approaches becomes apparent in the
 multivariate case and we confirm that they indeed offer different numerical schemes for
 estimating the  regularity of subdivision. We illustrate our results with several examples.
\end{abstract}


\section{ Introduction}
Subdivision schemes are recursive algorithms that starting with a coarser mesh in $\RR^s$ determine  the coordinates  of the finer
vertices $\cb^{(r+1)}$ by local averages of the
coarser ones
\begin{equation} \label{def:subrecursion}
 \cb^{(r+1)}=S_{\Ab} \cb^{(r)}, \quad r \ge 0.
\end{equation}
The subdivision operator $S_\Ab$ is linear and describes the local
averaging rules. The locality of the subdivision and algorithmic
simplicity of the subdivision recursion ensure that
\eqref{def:subrecursion} is fast, efficient, and easy to
implement. These features make subdivision popular in computer
graphics and animation, see  \cite{CDM91, Chui, DLsurvey, PReif}
and the references therein.

 We restrict our study to the shift--invariant setting and study the multivariate
vector schemes defined by the subdivision operator
$$
  S_{\Ab} \cb^{(r)}=\sum_{\beta \in \ZZ^s} A(\cdot-2\beta) c^{(r)}(\beta),
  \quad r \ge 0,
$$
mapping the space $\ell^n(\ZZ^s)$ of vector--sequences indexed by
$\ZZ^s$ into itself. The associated subdivision mask $\Ab=\left(
A(\alpha) \right)_{\alpha \in [0,N]^s}$, $N \in \NN$,
is a finitely supported matrix-valued sequence.

A challenging task is to provide a characterization of the
regularity of the limits of subdivision recursion and to develop
numerical methods for checking their regularity. Two prominent
methods that characterize the $W_p^k$-regularity, $1 \le p \le
\infty$, $k \in \NN$, of subdivision in the shift-invariant
setting are the so-called joint spectral radius (JSR) approach
\cite{CJR02, J95} and restricted spectral radius (RSR)
\cite{CDM91, CCS04} approach. Generally speaking, the joint
spectral radius approach characterizes the regularity of
subdivision limits in terms of the joint spectral radius of a
finite set
$$
 \cA=\{A_\varepsilon|_{V_k} \ : \ \varepsilon \in \{0,1\}^s\}
$$
of square matrices $A_\varepsilon$ derived from the mask $\Ab$ and
restricted to a common finite dimensional invariant subspace
$V_k$. The sequences in $V_k$ annihilate certain polynomial
eigensequences  of $S_\Ab$ and, thus, the process of restricting  of
$A_\varepsilon$ to $V_k$ mimics the computation of discrete
derivatives $\nabla^k$. The essence of the restricted spectral
radius approach can be formulated as follows: if the operator
$S_\Ab$ admits the factorization
\begin{equation} \label{iden:AnablaB}
 \nabla^k S_{\Ab}=S_{\Bb_k} \nabla^{k}, \quad k \ge 1,
\end{equation}
then the restricted spectral properties of the associated
difference subdivision schemes $S_{\Bb_k}$ characterize the
regularity of the underlying subdivision.
It had been believed until recently that the JSR and RSR
approachers are intrinsically different. In \cite{C2011, CCS05},
we showed that these two approaches characterize the
$W_k^p-$regularity, $1 \le p \le \infty$, of subdivision in terms
of the same quantity $(k,p)-$JSR and differ by the numerical
methods they yield for its estimation. Motivated by applications
in computer graphics and animation we compare in this paper only
the numerical methods for approximation of $(k,\infty)-$JSR. The problem of
computing the $(k,\infty)-$JSR is NP-hard, i.e, there is no
polynomial--time algorithm for its approximation, see \cite{T}.
The brute force Branch--and-Bound algorithm dating back to
\cite{DL1992} is available.  Its approximation  of
$(k,\infty)-$JSR is based on the estimate
\begin{equation} \label{algorithm:branch_and_bound}
 \max_{A_{\varepsilon_j}|_{V_k} \in \cA|_{V_k}} \left(\rho \left(\prod_{j=1}^r
 A_{\varepsilon_j}|_{V_k} \right)\right)^{1/r} \le \rho_\infty(\cA) \le
 \max_{A_{\varepsilon_j}|_{V_k} \in \cA|_{V_k}}  \left\|
 \prod_{j=1}^r  \right\|^{1/r},
\end{equation}
$r \in \NN$, and allows us to approximate $(k,\infty)-$JSR with an
arbitrary precision.   The estimate in \eqref{algorithm:branch_and_bound} is
independent of the choice of the matrix norm $\| \cdot \|$. The Branch--and-Bound algorithm is
computationally expensive and its rate of convergence is not
known. See \cite{GPro, HReif} for a successful attempt to reduce
the computational complexity of the Branch--and-Bound algorithm. The restricted spectral radius
approach leads to linear programming, see \cite{C2011}, that due
to the usual properties of the spectral radii  yields another
upper bound for $(k,\infty)-$JSR. The computations in \cite{Os}
show that the lower bound in \eqref{algorithm:branch_and_bound} is
not always reliable. 

The main goal of this paper is to compare the numerical methods
arising from the joint and restricted spectral radius approaches
for approximation the $(k,\infty)-$JSR. In section
\ref{subsec:univ} we show that if we choose the matrix norm $\| \cdot \|$ in
\eqref{algorithm:branch_and_bound} to be the infinity matrix norm,
then the upper bound in \eqref{algorithm:branch_and_bound} and the
one obtained using the optimization approach in \cite[Section
4.6]{C2011} coincide in the univariate scalar and vector cases.
Note that the flexibility of the joint spectral radius approach is
that the estimate in \eqref{algorithm:branch_and_bound} is
independent of the choice of the matrix norm $\| \cdot \|$.  This
in some cases leads to  a sharper upper estimate for
$(k,\infty)-$JSR, see example \ref{ex1} or even allows for exact
computations of the $(k,\infty)-$JSR, see \cite{GPro, HReif}. The results
of section \ref{subsec:univ} also show how one can easily determine the entries 
of the matrices $ A_{\varepsilon_j}|_{V_k}$ from the entries of the
corresponding difference maks $\Bb_k$. 
In section
\ref{subsec:multi} we show that in the multivariate scalar or
vector cases the Branch--and-Bound algorithm and the optimization
method presented in \cite[Section 4.6]{C2011}  are intrinsically
different and yield different upper bounds for $(k,\infty)-$JSR
regardless of the choice of the matrix norm in
\eqref{algorithm:branch_and_bound}. In Section
\ref{example:divergent}, we compare the method for estimation of
the $(k,\infty)-$JSR from \cite{LW95} and the one based on the
properties of the difference masks on an example of a divergent
scheme. To summarize, examples illustrate that it is impossible to
prefer one of  
the numerical methods discussed in this paper  over the other  and they
can be used according to one's personal preference.

\section{Notation and Background}

 An element $\mu=(\mu_1,\dots,\mu_s) \in
\NN_0^s$ is a multi--index whose length is given by
$|\mu|:=\mu_1+\dots + \mu_s$ and $\mu!:=\mu_1! \dots \mu_s!$.
 For $\alpha=(\alpha_1, \dots, \alpha_s) \in \ZZ^s$ and
$\mu=(\mu_1, \dots, \mu_s) \in \NN^s_0$ define
$$
 \alpha^\mu:=\alpha_1^{\mu_1} \dots \alpha_s^{\mu_s}.
$$
Denote by $\epsilon_\ell$, $\ell=1,\dots,s$,  the $\ell-$th
standard unit vector of $\RR^s$ and by $e_j$, $j=1, \dots,  n$,
the $j-$th standard unit vector of $\RR^n$, respectively. Let
$\ell^{n \times k} \left( \ZZ^s \right)$ denote the linear space
of all sequences of $n \times k$ real matrices indexed by $\ZZ^s$.
In addition, let $\ell^{n \times d}_\infty \left( \ZZ^s \right)$
denote the Banach space of sequences of $n \times d$ real matrices
indexed by $\ZZ^s$ with finite $\infty$-norm defined as
\begin{equation}\label{L-p-norms}
 \|\Cb \|_\infty  :=\displaystyle{ \sup_{\alpha \in \ZZ^s} \left| C
(\alpha) \right|_\infty},
\end{equation}
where $|C(\alpha)|_\infty$ is the $\infty$-operator norm on
$\RR^n$ if $d>1$ and the $\infty$-vector norm if $d=1$.  Moreover,
let $\ell_0^{n\times d} \left( \ZZ^s \right) \subset
\ell_\infty^{n \times d}(\ZZ^s)$ be the space of finitely
supported matrix valued sequences. Specific examples of such
scalar and vector sequences is the scalar delta sequence
$\mbox{\boldmath $\delta$}\in \ell_0(\ZZ^s)$ and the vector
sequence $\mbox{\boldmath $\delta$} e_j^T \in \ell^{1 \times n
}_0(\ZZ^s)$ defined by
\begin{equation}
 \delta(\alpha):=\left\{
 \begin{array}{cl} 1, & \alpha = 0, \\
 0, & \alpha \in \ZZ^s \setminus \{ 0 \}
 \end{array} \right. \quad \hbox{and} \quad
 \delta e_j^T(\alpha):=\left\{
 \begin{array}{cl} e_j^T, & \alpha = 0, \\
 0, & \alpha \in \ZZ^s \setminus \{ 0 \}.
 \end{array} \right.
\end{equation}

Let $\Ab \in \ell_0^{n \times n}(\ZZ^s)$ be a finitely supported
matrix sequence shifted so that $\supp \Ab \subseteq [0,N]^s$ for
some  $N \in \NN$.
The \emph{subdivision operator} $S_\Ab: \ell^{n} \left( \ZZ^s
\right)\rightarrow \ell^{n} \left( \ZZ^s \right)$ associated with
the \emph{mask} $\Ab \in \ell_0^{n \times n}(\ZZ^s)$ is defined by
\begin{equation} \label{def:S_a}
 S_\Ab\cb(\alpha)=\sum_{\beta\in
 \ZZ^s}A(\alpha-2\beta)c(\beta),\quad \alpha \in \ZZ^s.
\end{equation}
\noindent The \emph{subdivision scheme}
then corresponds to a repeated application of $S_\Ab$ to an
initial vector sequence $\cb \in\ell^{n } \left( \ZZ^s \right)$
yielding
\begin{equation}\label{subdivision}
\cb^{(0)}:=\cb, \quad \cb^{(r+1)}:=S_\Ab\ \cb^{(r)},\quad r\ge 0.
\end{equation}
The dimension of the following subspace of $\RR^n$ 
\begin{equation} \label{def:E_A}
\cE_\Ab := \left\{ v \in \RR^n \;:\;  \sum_{\alpha \in \ZZ^s}
A(\varepsilon-2\alpha) v=v, \, \varepsilon \in \{0,1\}^s \right\}
\end{equation}
determines the structure
of the difference operators $\nabla^k$ we define next.
Let $m:=\hbox{dim}(\cE_\Ab)$. For $\Cb \in \ell^{n \times
d}(\ZZ^s)$ and $\Db \in \ell^{d \times n}(\ZZ^s)$ we define the
$j-$th column  of $\nabla_\ell \Cb$, $\nabla_\ell : \ell^{n \times
d}(\ZZ^s) \rightarrow \ell^{n \times d}(\ZZ^s)$ as
\begin{equation} [\nabla_\ell \Cb]_{\cdot,j}:= \left[ \begin{array}{c}
-\Cb_{1,j}+\Cb_{1,j}(\cdot-\epsilon_\ell) \\
\vdots \\
-\Cb_{m,j}+\Cb_{m,j}(\cdot-\epsilon_\ell) \\
\Cb_{m+1,j} \\ \vdots \\
\Cb_{n,j}\end{array} \right], \quad 1 \le \ell \le s, \quad 1 \le
j \le d,
\end{equation}
 and $\nabla_\ell \Db:= \left(
\nabla_\ell \Db^T\right)^T$, respectively. Let $k \in \NN$.
 For our analysis we make use of the notion of the
$k-$th difference operator, the discrete analog of a
derivative. The $k-$th \emph{difference operator} $\nabla^k
\;:\; \ell^{n \times d} \left( \ZZ^s \right) \to \ell^{n N_{s,k}
\times d} \left( \ZZ^s \right)$, \small
$N_{s,k}=\left(\begin{array}{c} s+k-1 \\ s-1 \end{array} \right)$,
\normalsize is defined by
\begin{equation} \label{def:nabla}
\nabla^k: =\left[ \begin{array}{cc}
 \nabla_1^{\mu_1} \dots \nabla_s^{\mu_s} I_m & 0 \\ 0  & I_{n-m}
 \end{array}\right]_{|\mu| = k}.
\end{equation}

 We say that the subdivision scheme $S_\Ab$ is {\it
$C^k-$convergent}, if for any starting sequence $\cb \in
\ell_\infty^n(\ZZ^s)$ there exists a vector--valued function
$f_\cb \in (C^k(\RR^s))^n$ such that for any compactly supported
stable test function $g \in C^k(\RR^s) $
\begin{equation}
\label{def:convergenceFunction}
 \lim_{r \rightarrow \infty} \max_{|\mu| \le k} \|D^\mu f_\cb -
 g I_n * (2^{|\mu| r} \nabla_1^{\mu_1} \dots \nabla_s^{\mu_s} S_\Ab^r \cb)(2^r \cdot)
 \|_{\infty} =0.
\end{equation}
For more details on the properties of test functions see
\cite{DM97}.

Another useful tool for studying subdivision schemes is the
Laurent polynomial formalism. For a finite matrix sequence $\Ab
\in \ell_0^{n\times k}(\ZZ^s)$ we define the associated
\emph{symbol} as the \emph{Laurent polynomial}
\begin{equation}
 A^* (z):=2^{-s} \sum_{\alpha \in \ZZ^s} A(\alpha) \, z^{\alpha},
 \qquad z \in \left( \CC \setminus \{0\} \right)^s,
\end{equation}
where, in the usual multi--index notation,
$z^{\alpha}=z_1^{\alpha_1} \cdot \ldots \cdot z_s^{\alpha_s}$.

The concept of the joint spectral radius  was first introduced in
\cite{RS60} and was rediscovered in \cite{DL1992} for
characterizing the continuity of univariate subdivision limits.
For extensions of the concept of the JSR  see \cite{CJR02, J95}.
Generally speaking, the joint spectral radius approach
characterizes the $W^k_p-$regularity of subdivision limits in
terms of $(k,p)-$joint spectral radius (JSR)  of a finite set of linear operators. These are
derived from the linear operators $\cA_\varepsilon: \ell_0^{1
\times n}(\ZZ^s) \rightarrow \ell_0^{1 \times n}(\ZZ^s)$
satisfying
$$
 \cA_\varepsilon \vb = \sum_{\alpha \in \ZZ^s} v(\alpha) A(\varepsilon+2 \cdot -\alpha),
 \quad \vb \in \ell_0^{1 \times n}(\ZZ^s), \quad \varepsilon \in \{0,1\}^s,
$$
and then restricted to the invariant subspace $V_k \subset
\ell_0^{1 \times n}(\ZZ^s)$. The elements of the finite
dimensional subspace
\begin{eqnarray}\label{def:Vvector}
 V_k&:=&\{\vb \in \ell^{1 \times n}([0,N]^s): \sum_{\beta \in \ZZ^s} v(\beta) u(-\beta)=0 \
  \hbox{for all} \ \ub \in U_k\},
\end{eqnarray}
annihilate the polynomial sequences in $U_k \subset
\ell^{n}(\ZZ^s)$  reproduced by one step of subdivision recursion
in \eqref{subdivision}. For details on the structure of $U_k$ see
\cite{C2011}. Thus, the process of restricting to the subspace
$V_k$ imitates the computation of the partial derivatives of
subdivision limits function, see \cite[Section 2.3]{CHM04}. As the
subspace $V_k$ is finite dimensional we replace the finite set of
operators
\begin{equation}\label{def:ArestrictedV}
 \cA|_{V_k}:= \left\{ \cA_\varepsilon|_{V_k} \,:\, \varepsilon \in
 \{0,1\}^s  \right\}
\end{equation}
by a finite set of square matrices
\begin{equation} \label{def:Aepsilon_matrix}
 A_\varepsilon:=\left[ A^T(\varepsilon+2\alpha-\beta) \right]_{\alpha, \beta \in [0,N]^s},
 \quad \varepsilon \in \{0,1\}^s,
\end{equation}
and use them for  estimations of the $(k,\infty)-$joint spectral
radius
\begin{equation} \label{def:JSR}
  \rho_\infty(\cA|_{V_k}):=\displaystyle{\lim_{r \rightarrow \infty}
         \max_{A_{\varepsilon_j}|_{V_k} \in \cA|_{V_k}}  \left\|
 \prod_{j=1}^r A_{\varepsilon_j}|_{V_k} \right\|^{1/r}}.
\end{equation}


For a difference scheme $ S_{\Bb_k}$, $k \ge 1$, satisfying
\eqref{iden:AnablaB} define the {\it restricted $(k,\infty)$-norm}
\begin{eqnarray}\label{def:restricted_norm}
 \|S_{\Bb_k} |_{\nabla^{k}}\|_\infty:&=&  \sup \left \{ {\| S_{ \Bb_k} \nabla^{k} \cb\|_\infty
 \over \|\nabla^{k} \cb\|_\infty}: \  \cb \in  \ell_\infty^n(\ZZ^s), \nabla^{k} \cb \not=0
\right\}
\end{eqnarray}
and the {\it restricted $(k,\infty)-$spectral radius} ($(k,\infty)$-RSR)
\begin{equation} \label{def:RSR}
 \rho_\infty(S_{ \Bb_k}|_{\nabla^{k}}):= \lim_{r \rightarrow \infty} \|S^r_{
\Bb_k}|_{\nabla^k}\|_\infty^{1/r}.
\end{equation}
For the results on existence of $S_{\Bb_k}$ see e.g. \cite{CDM91,C2011}.
To estimate the $(k,\infty)-$RSR we use the standard properties of
spectral radii
$$
   \rho_\infty(S_{\Bb_{k}}|_{\nabla^{k}})=\inf_{r \in \NN} \|S_{\Bb_k}^r
   |_{\nabla^k}\|_\infty^{1/r}
$$
and get the following estimate
\begin{equation} \label{estimate:RSR}
   \rho_\infty(S_{\Bb_{k}}|_{\nabla^{k}}) \le  \|S_{\Bb_k}^r |_{\nabla^k}\|_\infty^{1/r}
\end{equation}
for any $r \in \NN$. The restricted norm $\|S_{\Bb_k}^r
|_{\nabla^k}\|_\infty$ can be computed using linear programming in
\cite[Section 4.6]{C2011}. The main result of \cite{C2011} shows that
$(k,p)-$JSR and $(k,p)-$RSR are equal. 
 
\section{Methods for approximating $(k,\infty)-$JSR} \label{sec:comparison}

In this section we show how the entries of the matrices
$A_{\varepsilon_j}|_{V_k}$ and their products depend on the
entries of the difference masks $\Bb^{(r)}_k$, $r \ge 1$ and $1 \le k <
N$. In particular, we show that  the upper bounds in
\eqref{algorithm:branch_and_bound} and \eqref{estimate:RSR} on the
$(k,\infty)-$joint spectral radius coincide in scalar and vector
univariate cases for the special choice of the matrix norm, i.e.
\begin{equation} \label{identity:univ_equality}
 \|S_{\Bb_k}^r |_{\nabla^k}\|_\infty=\|S_{\Bb_k}^r \|_\infty= \max_{A_{\varepsilon_j}|_{V_k} \in \cA|_{V_k}}  \left\|
 \prod_{j=1}^r A_{\varepsilon_j}|_{V_k} \right\|_\infty.
\end{equation}
Varying the matrix norm in \eqref{algorithm:branch_and_bound} can
lead to a better or a worse upper estimates of the joint spectral
radius, see Example \ref{ex1}. In the scalar or vector
multivariate case, the identity
$$
 \|S_{\Bb_k}^r \|_\infty= \max_{A_{\varepsilon_j}|_{V_k} \in \cA|_{V_k}}  \left\|
 \prod_{j=1}^r A_{\varepsilon_j}|_{V_k} \right\|_\infty
$$
does not hold in general, see subsection \ref{subsec:multi}.
Therefore, the JSR and RSR approaches lead to intrinsically
different numerical methods for checking regularity of
subdivision.

An important observation that connects the entries of the matrices
$A_{\varepsilon_j}|_{V_k}$ with the ones of the difference masks
$\Bb_k$, $k \ge 1$, states the following.

\begin{Proposition} \label{prop:main} For $k,r \ge 1$, $\varepsilon_j \in \{0,1\}^s$ and $\widetilde{\beta} \in \ZZ^s$
we have
$$
  \cA_{\varepsilon_r} \dots \cA_{\varepsilon_1} \nabla^k \bdelta I_n(\cdot-\widetilde{\beta})= \sum_{\beta \in \ZZ^s}
  B^{(r)}_{k}(\varepsilon_1+\dots +
  2^{r-1}\varepsilon_r+2^r\beta- \widetilde{\beta}) \nabla^k
  \bdelta I_n(\cdot-\beta) .
$$
\end{Proposition}
\begin{proof} Let $r \ge 1$ and $\widetilde{\beta} \in \ZZ^s$.
By \cite[Lemma 2.2]{HJ98} we have
$$
 \nabla^k S_\Ab^r \bdelta I_n(\alpha-\widetilde{\beta})=
 \cA_{\varepsilon_1} \dots \cA_{\varepsilon_r} \nabla^k
 \bdelta I_n(\gamma-\widetilde{\beta}),
$$
where $\alpha=\varepsilon_1+\dots+ 2^{r-1}\varepsilon_1+2^r
\gamma$, $\gamma \in \ZZ^s$. Note that the application of
$\cA_{\varepsilon}$  to $\nabla^k \bdelta I_n$ means that the operator
$\cA_{\varepsilon}$  is applied to each row sequence of $\nabla^k
\bdelta I_n$ separately. Next using the identity \eqref{iden:AnablaB}
we obtain
$$
 S_{\Bb_k}^r \nabla^k \bdelta I_n (\alpha-\widetilde{\beta}) =\cA_{\varepsilon_1} \dots \cA_{\varepsilon_r} \nabla^k
 \bdelta I_n(\gamma-\widetilde{\beta}), \quad \gamma \in \ZZ^s.
$$
The definition of $S^r_{\Bb_k}$  and the definition of the
standard convolution operator $*$ yield that the elements of the
sequence $\cA_{\varepsilon_1} \dots \cA_{\varepsilon_r} \nabla^k
\bdelta I_n(\cdot-\widetilde{\beta})$ have the following
representation
$$
 \cA_{\varepsilon_1} \dots \cA_{\varepsilon_r} \nabla^k
 \bdelta I_n(\gamma-\widetilde{\beta})= \left(
 \nabla^k \bdelta I_n * \Bb_k^{(r)}(\varepsilon_1+\dots +
  2^{r-1}\varepsilon_r+2^r\cdot -\widetilde{\beta}) \right)(\gamma), \quad
  \gamma \in \ZZ^s,
$$
in terms of the elements of the subsequence
$\Bb_k^{(r)}(\varepsilon_1+\dots +
  2^{r-1}\varepsilon_r+2^r\cdot -\widetilde{\beta})$ of $\Bb^{(r)}_k$.
Thus, the claim follows.
\end{proof}

\subsection{Univariate scalar and vector cases} \label{subsec:univ}

Let us see how to use the result of Proposition \ref{prop:main} to
compare the upper bound in \eqref{algorithm:branch_and_bound} and
\eqref{estimate:RSR}. For that we need the following auxiliary
lemma.

\begin{Lemma} \label{lemma:restricted=nonrestricted}
For $r \ge 1$, $1 \le k < N$, and $S^{r}_{\Bb_k} : \ell^{n}(\ZZ)
\rightarrow \ell^n(\ZZ)$ we have
$$
 \|S^{r}_{\Bb_k}
 |_{\nabla^k}\|_\infty=\|S^{r}_{\Bb_k}\|_\infty.
$$
\end{Lemma}
\begin{proof} We consider the vector case first.
Due to \cite[Proposition 6.10]{D92}, the non-restricted operator
$(k,\infty)-$norm $\|S^{r}_{\Bb_k}\|_\infty$ is defined as usual by
$$
 \|S^{r}_{\Bb_k}\|_\infty=\max \left\{  \|S^r_{\Bb_k} \nabla^k \cb \|_\infty \ :
 \ \cb \in \ell_\infty^n(\ZZ), \ \|\nabla^k \cb\|_\infty=1 \right\}
$$
and can be  computed as follows
\begin{equation} \label{inftynorm}
\|S^r_{\Bb_k}\|_\infty = \max_{\alpha \in \{0,2^r-1\}} \left\{
  \left| \sum_{\beta \in \ZZ} |B^{(r)}_{k}(\alpha-2^r\beta)| \right|_\infty
  \right\}.
\end{equation}
The absolute value of the matrices $B^{(r)}_{k}$ in
\eqref{inftynorm} means that we take the absolute values of each
matrix entries. Note that, due to the compact support of $\Ab$ and the definition of $\nabla^k$, we
have $\hbox{supp}(\Bb^{(r)}_k) \subset [0,N+k]$ for all $r \ge 1$.
Thus, \eqref{inftynorm} is equivalent to the following linear
programming: for $i=1, \dots ,n$ and $\alpha \in \{0,2^r-1\}$
\begin{equation} \label{eq:NormInfty}
\begin{array}{c} \displaystyle{\max \sum_{\beta\in  [-N-k,0]}
\sum_{j=1}^{n}
   (B^{(r)}_{k})_{ij} \left( \alpha - 2^r \beta \right) d_j(\beta)} \\ \\
   -1\le d_j(\beta) \le 1, \quad \beta\in  [-N-k,0], \quad j=1, \dots,
   n.
   \end{array}
\end{equation}
For fixed $i$ and $\alpha$, the solution $\db^{(i,\alpha)} \in
\ell^{n}_\infty([-N-k,0])$ of this optimization problem is given
by
$$
 d^{(i,\alpha)}_j(\beta)=\hbox{sgn}(B_k^{(r)})_{ij} \left( \alpha - 2^r \beta
\right), \quad \beta \in [-N-k,0].
$$
We get the maximizing sequence  $\db \in
\ell^{n}_\infty([-N-k,0])$ for all $i$ and $\alpha$ choosing
$\db=\db^{(i,\alpha)}$ such that the number determine by linear
programming in \eqref{eq:NormInfty} for this $\db$ is maximal.

On the other hand, from \cite{C2011} the computation of the
restricted $(k, \infty)-$norm is equivalent to the following
linear programming: for $i=1, \dots ,n$ and $\alpha \in
\{0,2^r-1\}$
\begin{equation} \label{eq:LinOptInfty}
\begin{array}{c} \displaystyle{\max \sum_{\beta\in  [-N-k,0]}
\sum_{j=1}^{n}
   (B^{(r)}_{k})_{ij} \left( \alpha - 2^r \beta \right) \nabla^k c_j(\beta)} \\ \\
   -1\le \nabla^k c_j(\beta) \le 1, \quad \beta\in  [-N-k,0], \quad j=1, \dots,
   n.
   \end{array}
\end{equation}
For fixed $i$ and $\alpha$, the solution $\cb^{i,\alpha} \in
\ell^{n}_\infty([-N-2k,0])$ of this optimization problem is given
for each $j$ by  the linear system of equations $\nabla^k
c^{i,\alpha}_j(\beta)=\hbox{sgn}(B_k^{(r)})_{ij} \left( \alpha -
2^r \beta \right)$, $\beta \in [-N-k,0]$. The solutions $c_j$
exist and are unique, due to the invertibility of the
corresponding matrix of this system, which is bi-diagonal with
$-1$ and $1$ on the main and upper diagonals, respectively. We get
the maximizing sequence  $\cb \in \ell^{n}_\infty([-N-2k,0])$ for
all $i$ and $\alpha$ choosing $\cb=\cb^{(i,\alpha)}$ such that the
number determine by linear programming in \eqref{eq:LinOptInfty}
for this $\cb$ is maximal. Therefore, \eqref{lemma:restricted=nonrestricted} is
satisfied.

In the scalar univariate case the proof is analogous, one should
only take into account that the support of $\Bb_k$ decreases with
$k$ and replace above $[0,N+k]$ by $[0,N-k]$ and $[-N-k,0]$ by
$[-N+k,0]$, respectively.
\end{proof}

Combining Proposition \ref{prop:main} and Lemma
\ref{lemma:restricted=nonrestricted} we obtain the following
result.

\begin{Theorem}  \label{Th:univ_equality} For $r \ge 1$ and $1 \le k
<N$ we have
\begin{equation} \label{eq:univ_equality}
 \max_{ \varepsilon_j \in \{0,1\}}  \Big \| \prod_{j=1}^r
 A_{\varepsilon_j}|_{V_k} \Big \|_\infty  = \left\|S^r_{B_k} |_{\nabla^k}
 \right\|_\infty.
\end{equation}
\end{Theorem}
\begin{proof} We start with the scalar case. Due to \cite[Lemma 4.4,
4.9]{C2011}, we have
\begin{equation} \label{Vk}
 V_k=\hbox{span}\left\{ \nabla^k \bdelta (\cdot -\beta) \ :\ \beta \in [0,N-k]
 \right\}.
\end{equation}
Thus, due to the fact that the spanning set above is also a basis for $V_k$ and by 
Proposition \ref{prop:main}, we get for any $r \ge 1$ and
$\varepsilon_j \in \{0,1\}$
$$
 A_{\varepsilon_r} \dots A_{\varepsilon_1}|_{V_k}=\Big[B^{(r)}_k(\varepsilon_1+\dots +
 2^{r-1}\varepsilon_r+2^r\widetilde{\beta}-\beta) \Big]_{\widetilde{\beta}, \beta \in
 [0,N-k]}.
$$
 Therefore,
$$
 \max_{ \varepsilon_j \in \{0,1\}}  \Big \| \prod_{j=1}^r
 A_{\varepsilon_j}|_{V_k} \Big \|_\infty=\|S^r_{B_k} \|_\infty
$$
and the claim follows by Lemma
\ref{lemma:restricted=nonrestricted}. 

In the vector case, define
as in \cite[Lemma 4.9]{C2011} the spanning set of $V_k$ to be
\small $$ 
  \left\{ \left(\nabla^k \bdelta  e_j\right)^T
  (\cdot -\beta) \ :\ \beta \in
  [0,N-k], \ j=1, \dots, m \right\} \bigcup   \left\{ \bdelta  e_j^T
  (\cdot -\beta) \ :\ \beta \in
  [0,N], \ j=m+1, \dots, n \right\}.
$$
\normalsize This set is also a basis for $V_k$. The corresponding matrix
representations of $\cA_\varepsilon|_{V_k}$ are obtain as follows.
First consider the set  $\widetilde{V}_k$ spanned by
\small $$ 
  \left\{ \left(\nabla^k \bdelta  e_j\right)^T
  (\cdot -\beta) \ :\ \beta \in
  [0,N], \ j=1, \dots, m \right\} \bigcup   \left\{ \bdelta  e_j^T
  (\cdot -\beta) \ :\ \beta \in
  [0,N], \ j=m+1, \dots, n \right\}.
$$
\normalsize
By Proposition \ref{prop:main} we get for any $r \ge 1$ and
$\varepsilon_j \in \{0,1\}$
$$
 A_{\varepsilon_r} \dots A_{\varepsilon_1}|_{\widetilde{V}_k}=\Big[B^{(r)}_k(\varepsilon_1+\dots +
 2^{r-1}\varepsilon_r+2^r\widetilde{\beta}-\beta) \Big]_{\widetilde{\beta}, \beta \in
 [0,N]}
$$
As the spanning set of $ \widetilde{V}_k$ is not a basis for $V_k$
we need to remove the rows and columns of
$A_{\varepsilon}|_{\widetilde{V}_k}$ that correspond to
$$
 \left\{ \left(\nabla^k \bdelta  e_j\right)^T
  (\cdot -\beta) \ :\ \beta \in
  [N-k+1,N], \ j=1, \dots, m \right\}
$$
Then the claim follows as in the scalar case.
\end{proof}

\begin{Remark}

(i) The size of the matrices $A_{\varepsilon}|_{V_k}$ in
\eqref{algorithm:branch_and_bound} is of importance for numerical
computations  of $(k, \infty)-$JSR and is determined by the dimension
of $V_k$. For a method that allows us to reduce the sizes of these
matrices see \cite{Protasov}.  In the scalar case, the restriction of
$A_{\varepsilon}$ to  $V_k$ in \eqref{Vk} is
equivalent to removing  trivial cycles of $A^*(z)$. 

(ii) Note that the choice of the matrix norm on the left hand-side
of \eqref{eq:univ_equality} is crucial. For any other matrix norm
the statement of the Theorem \ref{Th:univ_equality} is not true in
general, see examples below. This illustrates one of the
advantages of the numerical method arising from the JSR approach
and is exploited in \cite{GPro}, where the authors  approximate
the so-called extremal matrix norm that allows for exact
computations of the $(k,\infty)-$JSR.
\end{Remark}

Let us next illustrate the result of Theorem
\ref{Th:univ_equality} on some examples.

\begin{Example} \label{ex1}
Consider the $4-$point scheme with the mask given by
$$
 A^*(z)=\frac{1}{32} \left(-1+9z^2+16z^3+9z^4-z^6 \right)=
 \frac{1}{32} \left(1+z \right)^2 \left(-1+2z+6z^2+2z^3-z^4 \right).
$$
The corresponding subdivision scheme is $C^1$ \cite{HReif}. Due to
$$
 A^*(1)=1 \quad \hbox{and} \quad A^*(-1)=0,
$$
by \cite[Corollary 3.9]{C2011} we can compute the difference masks
$\Bb_1$ and $\Bb_2$ satisfying the equivalent formulation of  \eqref{iden:AnablaB}
$$
 (z-1)^k A^*(z)=B_k^*(z)(z^2-1)^k, \quad k=1,2,
$$ 
and given by
$$
  B^*_1(z)=
 \frac{1}{32} \left(1+z \right) B^*_2(z),  \quad B^*_2(z)=
 \frac{1}{32} \left(-1+2z+6z^2+2z^3-z^4 \right).
$$
By \cite[Lemma 2.4]{C2011} we have
$$
 V_1=\hbox{span} \Big\{ \nabla \bdelta (\cdot-\beta) \ : \ 0 \le \beta \le
 5
 \Big\}
$$
and
$$
 V_2=\hbox{span} \Big\{ \nabla^2 \bdelta (\cdot-\beta) \ : \ 0 \le \beta \le 4
 \Big\}.
$$
The corresponding matrix representations of
$\cA_{\varepsilon}|_{V_k}$, $k=1,2$, are
$$
 A_0|_{V_1}=\Big[B_1(0+2\widetilde{\beta}-\beta) \Big]_{\widetilde{\beta}, \beta
 \in \{0,5\}}= \frac{1}{16}\left[ \begin{array}{rrrrrr}-1&8&1&0&0&0   \\
   0&1&8&-1&0&0\\ 0&-1&8&1&0&0
  \\ \dots \\ 0&0&0&1&8&-1
\end{array}\right],
$$
$$
 A_1|_{V_1}=\Big[B_1(1+2\widetilde{\beta}-\beta) \Big]_{\widetilde{\beta}, \beta
 \in \{0,5\}}= \frac{1}{16}\left[ \begin{array}{rrrrrr}1&8&-1&0&0&0   \\
   -1&8&1&0&0&0\\ 0&1&8&-1&0&0
  \\ \dots \\ 0&0&1&8&-1&0
\end{array}\right],
$$
and
$$
 A_0|_{V_2}=\Big[B_2(0+2\widetilde{\beta}-\beta) \Big]_{\widetilde{\beta}, \beta
 \in \{0,4\}}= \frac{1}{16}\left[ \begin{array}{rrrrr}-1&6&-1&0&0   \\
   0&2&2&0&0\\ 0&-1&6&-1&0
  \\ \dots \\ 0&0&-1&6&-1
\end{array}\right],
$$
$$
 A_1|_{V_2}=\Big[B_2(1+2\widetilde{\beta}-\beta) \Big]_{\widetilde{\beta}, \beta
 \in \{0,4\}}= \frac{1}{16}\left[ \begin{array}{rrrrr}2&2&0&0&0   \\
   -1&6&-1&0&0\\ 0&2&2&0&0
  \\ \dots \\ 0&0&2&2&0
\end{array}\right].
$$
Hence, on the one hand we get
$$
 \|S_{B_1}\|_\infty=\max \Big\{ \|
 A_{0}|_{V_1}\|_\infty, \ \|
 A_{1}|_{V_1}\|_\infty \Big \}=\frac{5}{8}
$$
and
$$
 \|S_{B_2}\|_\infty=\max \Big\{ \|
 A_{0}|_{V_2}\|_\infty, \ \|
 A_{1}|_{V_2}\|_\infty \Big \}=\frac{1}{2}.
$$
On the other hand
$$
 \|S_{B_1}\|_\infty <\max \Big\{ \|
 A_{0}|_{V_1}\|_2, \ \|
 A_{1}|_{V_1}\|_2 \Big \}=0.7195 \dots
$$
and
$$
 \|S_{B_2}\|_\infty >\max \Big\{ \|
 A_{0}|_{V_2}\|_2, \ \|
 A_{1}|_{V_2}\|_2 \Big \}=0.467 \dots$$
\end{Example}

The second example is of a vector univariate subdivision scheme.

\begin{Example}
Let the mask $\Ab \in \ell_0^{2 \times 2}(\ZZ)$ be given by its
symbol
$$
A^*(z)=\left( \begin{array}{cc} 1/2 & 1/4 \\ 0 & 1/4 \end{array}
\right) +  \left( \begin{array}{cc} 1 & 3/4 \\ 0 & 1/2 \end{array}
\right)z+\left( \begin{array}{cc} 1/2 & 1/4 \\ 0 & 1/4 \end{array}
\right)z^2.
$$
Due to
$$
 A^*(1) \cdot (1,0)^T=(1,0)^T \quad \hbox{and} \quad  A^*(-1) \cdot
 (1,0)^T=(0,0)^T,
$$
by  \cite[Corollary 3.9]{C2011} there exist the difference mask
$\Bb_1$ satisfying the equivalent formulation of \eqref{iden:AnablaB}
$$
\left[\begin{array}{cc} z-1 & 0 \\ 0 &1 \end{array}\right] A^*(z)=
B^*(z) \left[\begin{array}{cc} z^2-1 & 0 \\ 0 &1 \end{array}\right]
$$
and given by
$$
 B^*_1(z)=\left( \begin{array}{rr} 1/2 & -1/4 \\ 0 & 1/4 \end{array}
 \right)+
 \left( \begin{array}{rr} 1/2 & -1/2 \\ 0 & 1/2 \end{array}
 \right)z+
 \left( \begin{array}{cc} 0 & 1/2 \\ 0 & 1/4 \end{array}
 \right)z^2+ \left( \begin{array}{cc} 0 & 1/4 \\ 0 & 0 \end{array}
 \right)z^3.
$$
Note that in the vector case, due to the definition of the
difference operator $\nabla$, the support of the difference mask
$\Bb_1$ is larger than that of $\Ab$. To compare the entries of
$A_\varepsilon |_{V_1}$ to the ones of $\Bb_1$, we first, as in
the proof of Theorem \ref{Th:univ_equality}, define $\widetilde{V}_1$
to be the set spanned by
$$
\left\{ \left(\nabla \bdelta  e_1\right)^T (\cdot -\beta) \ :\
\beta \in
  [0,2] \right\} \cup \left\{  \bdelta  e_2^T (\cdot
-\beta) \ :\ \beta \in
  [0,2]\right\}.
$$
Then we get
$$
  A_\varepsilon|_{\widetilde{V}_1}=\Big[B_1(\varepsilon+2\widetilde{\beta}-\beta)
 \Big]_{\widetilde{\beta},\beta
 \in [0,2]}, \quad \varepsilon \in \{0,1\}
$$
from which we obtain  $5 \times 5$ square matrices
$A_\varepsilon|_{V_1}$, $\varepsilon \in \{0,1\}$, by removing one but
last rows and one but last columns of the corresponding matrices
$A_\varepsilon|_{\widetilde{V}_1}$. Note that such a choice of
$A_\varepsilon|_{V_1}$, $\varepsilon \in \{0,1\}$, depends on the
structure of the basis of $V_1$ give by
$$
\left\{ \left(\nabla \bdelta  e_1\right)^T (\cdot -\beta) \ :\
\beta \in
  [0,1] \right\} \cup \left\{  \bdelta  e_2^T (\cdot
-\beta) \ :\ \beta \in
  [0,2]\right\}.
$$
And we finally  get
$$
 \|S_{\Bb_1}\|_\infty=\max \Big\{ \| A_{0}|_{V_1}\|_\infty, \ \|
 A_{1}|_{V_1}\|_\infty \Big \}=\frac{5}{4}.
$$
\end{Example}

\subsection{Multivariate scalar and vector cases} \label{subsec:multi}

The result of Theorem \ref{Th:univ_equality} does not hold in
general in the multivariate case. One of the reasons for that is
that the spanning set of $V_k$ in \cite[Lemma 4.9]{C2011} is not a basis for $V_k$, although
the representation from Proposition \ref{prop:main} is still
valid. Another reason is that the result of   Lemma
\ref{lemma:restricted=nonrestricted} does not hold in general in
the multivariate case, see \cite[Example 5.2]{C2011}. The
following example illustrates that the numerical methods arising
from the RSR and JSR approaches are indeed two different methods
for estimating the $(k,\infty)-$JSR.

\begin{Example}
Consider the scalar multivariate scheme $S_\Ab : \ell(\ZZ^2)
\rightarrow \ell(\ZZ^2)$ given by its symbol
$$
 A^*(z_1,z_2)=\frac{1}{8}\Big(1+z_1 \Big)\Big(1+z_2 \Big)\Big(1+z_1z_2 \Big).
$$
Due to
$$
 A^*(1,1)=1 \quad \hbox{and} \quad  A^*(e)=1, \quad e \in
 \{1,-1\}^2 \setminus \{1\},
$$
the first difference scheme satisfying the equivalent formulation of \eqref{iden:AnablaB}
$$
\left[\begin{array}{c} z_1-1 \\ z_2-1 \end{array}\right] A^*(z)=
B^*(z) \left[\begin{array}{cc} z_1^2-1 \\ z_2^2-1 \end{array}\right]
$$
exists 
and is given, e.g. by the mask
$$
\Bb_1= \frac{1}{2}
\begin{array}{cccc}
 & \left[ \begin{array}{cc} 1 & 0 \\ 0 & 0 \\ \end{array} \right]  \\ \\
\left[ \begin{array}{cc} 1 & 0 \\ 0 & 0 \\ \end{array} \right] & \left[
\begin{array}{cc} 1 & 0 \\ 0 & 1 \\ \end{array} \right] & \left[
\begin{array}{cc} 0 & 0 \\ 0 & 1 \\ \end{array} \right]  \\ \\
\left[ \begin{array}{cc} \bf{1} & \bf{0} \\ \bf{0} & \bf{1} \\
\end{array} \right] &
\left[ \begin{array}{cc} 0 & 0 \\ 0 & 1 \\ \end{array} \right] &
\end{array} \ .
$$
As all the entries of $\Bb_1$  are non--negative, we apply
\cite[Remark 4.13]{C2011} and get
$$
 \max_{\|\nabla \cb\|_\infty=1} \|S_\Bb \nabla
 \cb\|_\infty=\max_{\|\cb\|_\infty=1} \|S_\Bb \cb\|_\infty=1/2.
$$
Note that, using \cite[Proposition 2.9]{GMW}, we can even compute
$\rho_\infty(S_{\Bb_1} |_\nabla)$. To do that we view $S_{\Bb_1}$
as a scheme consisting of two scalar first difference schemes
$S_{\widetilde{\Bb}_1}$ and $S_{\widetilde{\widetilde{\Bb}}_1}$
given by the masks
$$
\widetilde{\Bb}_1=\begin{array}{ccccccc}\dots &0&0&0&0&0&\dots \\
\dots&0&0&1/2&0&0&\dots \\
\dots&0&1/2&1/2&0&0&\dots \\
\dots&0&\bf{1/2}&0&0&0&\dots \\
\dots &0&0&0&0&0&\dots
\end{array} \quad \quad
\widetilde{\widetilde{\Bb}}_1=\begin{array}{ccccccc}\dots &0&0&0&0&0&\dots \\
\dots&0&0&0&0&0&\dots \\
\dots&0&0&1/2&1/2&0&\dots \\
\dots&0&\bf{1/2}&1/2&0&0&\dots \\
\dots &0&0&0&0&0&\dots
\end{array}
$$
with the entry at bold at the position $(0,0)$. Then by
\cite[Proposition 2.9]{GMW} we get
$$
 \rho_\infty(S_{\Bb_1}|_\nabla)=\max \left\{ \rho\left([\widetilde{\Bb}_1(\alpha-2\beta)]_{\alpha, \beta \in
 \Omega}\right),  \ \rho\left([\widetilde{\widetilde{\Bb}}_1(\alpha-2\beta)]_{\alpha, \beta \in
 \Omega}\right) \right\},
$$
where $\Omega=[-2,2]^2 \cap \ZZ^2$ by \cite[p. 345]{GMW} is any associated
good set. A simple computation yields
$$
 \rho_\infty(S_{\Bb_1} |_\nabla)=\frac{1}{2}=\max_{\|\cb\|_\infty=1} \|S_\Bb
 \cb\|_\infty.
$$
Let us compare $\rho_\infty(S_{\Bb_1} |_\nabla)$ with the upper
estimate for $(0,\infty)-$JSR produced by
\eqref{algorithm:branch_and_bound}. We fix a basis of $V_1$ to be
the following set
$$
 \{\nabla_1 \bdelta(\cdot-\beta)\ : \ \beta \in \{0,1\} \times
 \{0,2\}\} \cup \{ \nabla_2 \bdelta(\cdot-\beta) \ : \ \beta=(1,0),(1,1)\}.
$$
Note that using Proposition \ref{prop:main} we get
$$
 \cA_{(1,0)} \nabla_2 \delta(\cdot-(1,0))= \frac{1}{2}
 \nabla_2 \delta.
$$
and
$$
 \frac{1}{2}
 \nabla_2 \delta= \frac{1}{2} \nabla_1 \delta-\frac{1}{2}
 \nabla_1 \delta(\cdot-(0,1))+\frac{1}{2} \nabla_2 \delta(\cdot-(1,0)),
$$
where $\nabla_2 \delta$ is not one of the above basis elements of
$V_1$. Thus, the corresponding row of $A_{(1,0)}|_{V_1}$ does not
consist of  the entries of the difference mask $\Bb_1$ as it
happens in the univariate case. For all other possible choices of
a basis of $V_1$ we get analogous structure for at leats one of
the rows of some of $A_\varepsilon|_{V_1}$. Therefore, we get
$$
 \max_{ \varepsilon_j \in \{0,1\}^2}  \Big \|
 A_{\varepsilon_j}|_{V_1} \Big \|_\infty = \frac{3}{2} >
 \|S_{B_1}\|_\infty = \|S_{B_1}|_{\nabla}\|_\infty =JSR_{\infty}.
$$
and also
$$
 \max_{ \varepsilon_j \in \{0,1\}^2}  \Big \|
 A_{\varepsilon_j}|_{V_1} \Big \|_2=1.188 \ldots >
 \|S_{B_1}|_{\nabla}\|_\infty.
$$
\end{Example}

Possible extensions of the result in \cite{Protasov} to the
multivariate  case  is currently under investigation. Such an
extension will not only allow  us to reduce the size of $
A_{\varepsilon_j}|_{V_k}$, but also may lead to different $V_k$
whose structure allows for better comparison of the RSR and JSR
approaches.

\subsection{Divergence of subdivision schemes} \label{example:divergent}
 This scalar bivariate example is taken from \cite{HJ98}. The mask is given by its
symbol
$$
 A^*(z_1,z_2)=\frac{1}{4} \left(\frac{1}{4}+z_1+\frac{3}{4}z_1^2+\frac{3}{4}z_2+z_1z_2+\frac{1}{4}z_1^2z_2 \right).
$$
Note that the mask satisfies  $A^*(1,1)=1$ and
$A^*(1,-1)=A^*(-1,1)=A^*(-1,-1)=0$. It has been shown in
\cite{BW92} that
\begin{equation}\label{prop:inftyJSR}
  \rho_\infty(\cA|_V)= \sup \left\{|\lambda|^{1/r}: \
    r>0, \ \lambda \in  \sigma \left( A_{\varepsilon_1}|_{V_1} \cdots
      A_{\varepsilon_r}|_{V_1} \right), \, \varepsilon_j \in
    \{0,1\}^s \right\},
\end{equation}
where $\sigma(M)$ denotes the spectrum, i.e., the set of all
eigenvalues, of the matrix $M$. In our case, one of the
eigenvalues of $A_{(0,0)}|_{V_1}$ is equal to one, see \cite[page
1192]{HJ98} for details, and we get
$$
 \rho_\infty(\cA|_{V_1}) \ge 1.
$$
This shows that the scheme is not uniformly convergent
\cite{CJR02}. To get the same conclusion using the restricted
spectral radius approach, we need to show that there exists $R \in
\NN$ such that for all $r
> R$ we have $\| S^r_{\Bb_1}|_\nabla\|_\infty \ge 1$, which is a tedious
task.  Another possibility to show that the scheme is not
uniformly convergent is to show that $S_{\Bb_1}$ does not converge
to a zero limit function. To do that determine first the
difference scheme $S_{\Bb_1}$ satisfying
$$
 \left[\begin{array}{c} z_1-1 \\ z_2-1 \end{array} \right]
 A^*(z_1,z_2)= B^*(z_1,z_2) \left[\begin{array}{c} z_1^2-1 \\ z_2^2-1 \end{array} \right]
$$
with the matrix--valued symbol $B^*(z_1,z_1)=\left(b^*_{i,j}
\right)_{i,j=1,2}$ with the entries
\begin{eqnarray}
 b^*_{11}(z_1,z_2)&=&\frac{1}{16}(1+3 z_1+3 z_2+z_1z_2), \notag \\
 b^*_{12}(z_1,z_2)&=&0,  \quad  \quad b^*_{21}(z_1,z_2)=\frac{1}{16}(2z_2-2), \notag \\
  b^*_{22}(z_1,z_2)
& =&\frac{1}{16}(z_1^2+4z_1+3). \notag
\end{eqnarray}
The symbol satisfies
$$
 B_1^*(1,1)=\frac{1}{4}\left[\begin{array}{cc} 2&0\\ 0 & 2  \end{array}
 \right]
$$
and does not indicate that the associated scheme is not zero
convergent. The problem is though that not all of the eigenvalues
of the sub-symbols
$$
 B^*_{1,\varepsilon}(z_1,z_2)=\sum_{\alpha \in \ZZ^2}
 B_1(\varepsilon-2\alpha)z^{\alpha}, \quad \varepsilon
 \in\{0,1\}^2,
$$
at $(1,1)$ are less than $1$, e.g
$$
 B_{1,(0,0)}(1,1)=\frac{1}{4}\left[ \begin{array}{rr} 1&0\\
 -2&4\end{array}\right].
$$
This violates a necessary condition for the convergence of
subdivision, see \cite{DM97}.  Thus, the spectral properties of
the difference scheme $S_{\Bb_1}$ also indicates that the scheme
is not uniformly convergent.


\noindent\hbox{\vtop{\hsize15pc\parindent0pt
Maria Charina\\
Fakult\"at f\"ur Mathematik\\
TU Dortmund\\
Vogelpothsweg 87\\
D--44227 Dortmund\\
{\tt maria.charina@uni-dortmund.de} }}

\end{document}